\documentclass[letterpaper, 10 pt, conference]{ieeeconf}  % Comment this line out
\IEEEoverridecommandlockouts                              % This command is only
\usepackage{graphics} % for pdf, bitmapped graphics files
\usepackage{epsfig} % for postscript graphics files
\usepackage{amsmath} % assumes amsmath package installed
\usepackage{amssymb}  % assumes amsmath package installed
\usepackage{color,soul} % to use highlight command

\soulregister\ref{7}
\soulregister\begin{7}
\soulregister\end{7}

\newtheorem{mylem}{Lemma}

\newtheorem{mythm}{Theorem}
\newcommand\numberthis{\addtocounter{equation}{1}\tag{\theequation}}
\newcommand{\R}{\mathbb{R}}
\newcommand{\N}{\mathbb{N}}
\newcommand{\s}{\mathbb{S}}

\title{\LARGE \bf
Stability Analysis of Parabolic Linear PDEs with Two Spatial Dimensions Using Lyapunov Method and SOS
}

\author{Evgeny Meyer and Matthew M. Peet% <-this % stops a space
\thanks{This work was supported by the National Science Foundation under Grants No. 1301851 and 1301660.}% <-this % stops a space
\thanks{Evgeny Meyer  is a Ph.D student with the School for Engineering of Matter, Transport and Energy, Arizona State University, Tempe, AZ, 85281 USA
        {\tt\small edmeyer@asu.edu}}%
\thanks{Matthew M. Peet is an assistant professor with the School for Engineering of Matter, Transport and Energy, Arizona State University, Tempe, AZ, 85281 USA
        {\tt\small mpeet@asu.edu}}%
}

\begin{document}

\maketitle
\thispagestyle{empty}
\pagestyle{empty}

%%%%%%%%%%%%%%%%%%%%%%%%%%%%%%%%%%%%%%%%%%%%%%%%%%%%%%%%%%%%%%%%%%%%%%%%%%%%%%%%
\begin{abstract}

In this paper, we address stability of parabolic linear Partial Differential Equations (PDEs).
We consider PDEs with two spatial variables and spatially dependent polynomial coefficients.
We parameterize a class of Lyapunov functionals and their time derivatives by polynomials and express stability as optimization over polynomials.
We use Sum-of-Squares and Positivstellensatz results to numerically search for a solution to the optimization over polynomials. 
We also show that our algorithm can be used to estimate the rate of decay of the solution to PDE in the $L_2$ norm.
Finally, we validate the technique by applying our conditions to the 2D biological KISS PDE model of population growth and an additional example.
\end{abstract}

%%%%%%%%%%%%%%%%%%%%%%%%%%%%%%%%%%%%%%%%%%%%%%%%%%%%%%%%%%%%%%%%%%%%%%%%%%%%%%%%
\section{INTRODUCTION}

Stability analysis and controller design for Partial Differential Equations (PDEs) is an active area of research \cite{c1}, \cite{c2}.
One approach to stability analysis of PDEs is to approximate the PDEs with Ordinary Differential Equations (ODEs) using, e.g. Galerkin's method or finite difference, and then apply finite-dimensional optimal control methods, e.g. \cite{c3}--\cite{c5p}.
In this paper we consider stability analysis without discretization.
Specifically, we use Sum-of-Squares (SOS) optimization to construct Lyapunov functionals for parabolic PDEs with two spatial variables and spatially dependent coefficients.

It is well-known that existence of a Lyapunov function for a system of ODEs or PDEs is a sufficient condition for stability.
For example,~\cite{c6} uses a Lyapunov approach and Linear Operator Inequalities (LOIs) to provide sufficient conditions for exponential stability of a controlled heat and delayed wave equations.
Another method based on Lyapunov and semigroup theories was applied in \cite{c7} for analysis of wave and beam PDEs with constant coefficients and delayed boundary control.
In \cite{c8} a Lyapunov based analysis of semilinear diffusion equations with delays gave stability conditions in terms of Linear Matrix Inequalities (LMIs).
Extensive examples of applying the backstepping method to the boundary control of PDEs can be found in \cite{c9}--\cite{c13}.
Briefly speaking, backstepping uses a Volterra operator to search for an invertible mapping from the original PDE to a chosen "target" PDE, known to be stable.
In order to find such mapping, one has to solve analytically or numerically a PDE for Volterra operator's kernel.
If the mapping is found, it provides the boundary control law.
Applications for two-dimensional cases were discussed in \cite{c14} and \cite{c15}.
However, backstepping requires us to guess on the target PDE and solve the PDE for kernel, which may be a challenging task for PDEs with two spatial variables and spatially dependent coefficients. Moreover, backstepping cannot be used for stability analysis in the absence of a control input.

Note that SOS has been previously applied in \cite{c16} and \cite{c17} to find Lyapunov functionals for 1D parabolic PDE.
Input-Output analysis of PDE systems with SOS implementation is discussed in \cite{c18}.
Examples of using SOS in controller design for one-dimensional PDEs can be found in \cite{c19}--\cite{c21}.

The main contribution of this paper is to provide an algorithm for stability analysis of 2D parabolic PDEs with spatially varying coefficients.
We search for polynomial $s$ which defines a Lyapunov functional of the form\vspace{-.01in}
\[
V(u(t,\cdot))=\int_\Omega s(x)u(t,x)^2\,dx\vspace{-.01in}
\]
where $u$ is the solution of the PDE, $t$ represents time, $x$ and $\Omega$ are the spatial variable and domain.
We introduce stability conditions in the form of parameter dependent LMIs, which certify Lyapunov inequalities, i.e.\vspace{-.01in}
\begin{equation}
 V(u(t,\cdot))>0\text{ and } \frac{d}{dt}\left[V(u(t,\cdot))\right]<0\text{ for all}\ t>0.\vspace{-.01in}
\label{eq:Lyapu_ieq}
\end{equation}
There exist several algorithms for optimization over polynomials which can be applied to~\eqref{eq:Lyapu_ieq} in order to find $V$. See e.g.,~\cite{Reza_DCDS2015} for a survey on algorithms for optimization of polynomials. 
In this paper, we apply SOS and the Positivstellensatz results to find a solution to parameter dependent LMIs.
For details on Positivstellensatz see \cite{c22} or \cite{c23}.
We use the MATLAB toolbox SOSTOOLS (see \cite{c24} for details on SOSTOOLS) and SeDuMi, a well-known semidefinite programming solver, to solve the SOS optimization problem.

The paper is organized as follows.
In Section 2 we specify the notation and provide some background.
Lyapunov stability conditions for parabolic PDEs are given in Section 3.
We demonstrate the proposed method in Section 4 and present SOS optimization problems in Section 5.
Finally, we discuss numerical implementations in Section 6 and conclude the paper in Section 7 with ideas about future research steps.

%%%%%%%%%%%%%%%%%%%%%%%%%%%%%%%%%%%%%%%%%%%%%%%%%%%%%%%%%%%%%%%%%%%%%%%%%%%%%%%%
\section{PRELIMINARIES}

\subsection{Notation}

$\N$ is the set of natural numbers and $\N_0:=\N\cup\{0\}$.
$\R^n$ and $\s^n$ are the $n$-dimensional Euclidean space and space of $n\times n$ real symmetric matrices.
For $x\in\R^n$, let $x^T$ denote transposed $x$ and $x_i\in\R$ is the $i$-th component of $x$.
$\|\cdot\|_1$ is a norm on $\R^n$, defined as $\|x\|_1:=\sum_{i=1}^n|x_i|.$
For $X\in\s^n$, $X\leq0$ means that $X$ is negative semidefinite.
The symbol $*$ will denote the symmetric elements of a symmetric matrix.

For $\Omega\subseteq\R^n$ and $f:\Omega\to\R$ let $f(x)$ stand for $f(x_1,...,x_n)$ and $\int_\Omega f(x)\,dx$ represent an integral of $f$ over $\Omega$ with $dx:=dx_1dx_2...dx_n$.

Let $\N_0^n:=\{\alpha\in\R^n:\alpha_i\in\N_0\}$.
A vector $\alpha\in\N_0^n$ is called multi-index. For $l\in\N$ define the set
\[
Q_l^n:=\{\alpha\in\N_0^n : \|\alpha\|_1\leq l\}.\numberthis{\label{Q_l^n}}
\]
For $\alpha\in\N_0^n$, $x\in\R^n$ and $g:\R^n\rightarrow\R$ partial derivative
\vspace{-.05in}
\[
D^{\alpha}[g(x)]:=\frac{\partial^{\alpha}}{\partial x^{\alpha}}[g(x)]=\prod_{i=1}^{n}\frac{\partial^{\alpha_i}}{\partial x^{\alpha_i}_i}[g(x)].\numberthis{\label{parder}}\vspace{-.05in}
\]
Note that $\frac{\partial^0}{\partial x^0_i}[g(x)]=g(x)$ for any $i\in\{1,...,n\}$.
We will also use classical notation, $u_{x_1 x_2}(t,x):=\frac{\partial}{\partial x_2}[\frac{\partial}{\partial x_1}[u(t,x)]]$.

If for a function $f:\Omega\to\R$ and some $\alpha\in\N_0^n$ derivative $D^\alpha[f(x)]$ exists for all $x\in\Omega$, there exists $g:\Omega\to\R$ such that $g(x)=D^\alpha[f(x)]$ for all $x\in\Omega$.
For brevity $D^\alpha[f]:=g$.

Further we use $W^{2,2}(\Omega)$. It is one of Banach spaces $W^{k,p}(\Omega)$ which denote Sobolev spaces of functions $u:\Omega\rightarrow\R$ with $D^{\alpha}[u]\in L_p(\Omega)$ for all $\alpha\in Q^n_k$, where $Q^n_k$ is defined as in (\ref{Q_l^n}) and norm\vspace{-.05in}
\[
\Vert u\Vert_{k,p}:=\sum_{\|\alpha\|_1\leq k}\Vert D^{\alpha}[u]\Vert_{L_p},\vspace{-.075in}
\]
where $L_p(\Omega)$ stands for the space of Lebesgue-measurable functions $g:\Omega\rightarrow\R$ with norm, for $p\in\N$\vspace{-.05in}
\[
\Vert g\Vert_{L_p}:=\Big(\int_{\Omega}|g(s)|^{p}ds\Big)^{1/p}\vspace{-.05in}
\]
and $\Vert g\Vert_{L_\infty}:=\sup_{s\in\Omega}|g(s)|$.

It is known that for continuous functions $u:[0,\infty)\to W^{2,2}(\Omega)$ and $V:W^{2,2}(\Omega)\to\R$ their composition $(V\circ u):[0,\infty)\to\R$ is also continuous and the upper right-hand derivative $D^+_t V(u(t))$ is defined by
\[
D^+ [V(u(t))]:=\limsup_{h\to 0^{+}}\frac{V(u(t+h))-V(u(t))}{h}.
\]
Note that if $v:[0,\infty)\to\R$ is differentiable at $t\in(0,\infty)$ then $D^+[v(t)]=\frac{d}{dt}[v(t)]$.

\subsection{Polynomials}

For a multi-index $\alpha\in\mathbb{N}_0^n$ and $x\in\mathbb{R}^n$, let $x^\alpha:=\prod_{i=1}^nx_i^{\alpha_i}=x_1^{\alpha_1} x_2^{\alpha_2}...\,x_n^{\alpha_n}.$
Then $x^\alpha$ is a monomial of degree $\|\alpha\|_1\in\mathbb{N}_0$.
A polynomial is a finite linear combination of monomials $p(x):=\sum_\alpha p_\alpha x^\alpha$, where the summation is applied over a given finite set of multi-indexes $\alpha$ and $p_\alpha\in\R$ denotes the corresponding coefficient.
Let $\mathbb{R}[x]$ stand for polynomial ring in $n$ variables $x_1,...,x_n$ with coefficients in $\mathbb{R}$.
The degree of a polynomial $p\in\mathbb{R}[x]$ is the largest degree among all monomials, and is denoted by $\text{deg}(p)\in\mathbb{N}_0$.

A polynomial $p\in\mathbb{R}[x]$ is called Sum of Squares (SOS), if there is a finite number of polynomials $z_i\in\mathbb{R}[x]$ such that for all $x\in\mathbb{R}^n$, $p(x)=\sum_i z_i(x)^2$.
We denote the set of sum of square polynomials in variables $x$ by $\sum[x]$.
If $p\in\sum[x]$, then $p(x)\geq0$ for all $x\in\mathbb{R}^n$.

Polynomial matrices and SOS matrices are defined in a similar manner (See, e.g.~\cite{c25}).
We use $\sum[\s(x)]$ to denote the set of real symmetric SOS polynomial matrices. If $M\in\sum[\s^m(x)]$ then for all $x\in\R^n$, $M(x)\in\s^m$ and $M(x)\geq0$.

\vspace*{-0.05in}
\subsection{Comparison Lemma}

Recall the comparison principle from, e.g.~\cite{c26}.
\begin{mylem}
Consider the scalar differential equation $\frac{d}{dt}[u(t)]=f(t,u(t))$, $u(t_0)=u_0$, where $f(t,x)$ is continuous in $t$ and locally Lipschitz in $x$, for all $t\geq0$ and all $x\in J\subset\mathbb{R}$.
Let $[t_0,T)$ ($T$ could be infinity) be the maximal interval of existence of the solution $u$, and suppose $u(t)\in J$ for all $t\in[t_0,T)$.
Let $v$ be a continuous function whose upper right-hand derivative $D^+[v(t)]$ satisfies \vspace*{-0.05in}
\[
D^+[v(t)]\leq f(t,v(t)),\ \ v(t_0)\leq u_0\vspace*{-0.05in}
\]
with $v(t)\in J$ for all $t\in[t_0,T)$. Then, $v(t)\leq u(t)$ for all $t\in [t_0,T)$.
\end{mylem}

%%%%%%%%%%%%%%%%%%%%%%%%%%%%%%%%%%%%%%%%%%%%%%%%%%%%%%%%%%%%%%%%%%%%%%%%%%%%%%%%

\section{LYAPUNOV STABILITY FOR PARABOLIC PDE}

First consider the following general form of parabolic PDEs. For all $t\in(0,\infty)$ and $x\in\Omega\subset\mathbb{R}^n$,
\[
\hspace*{-0.2cm}u_t(t,x)=F(t,x,D^{\alpha(1)}\hspace*{-0.1cm}\left[u(t,x)\right],...,D^{\alpha(k)}\hspace*{-0.1cm}\left[u(t,x)\right]), \numberthis \label{th_pde}
\]
where $u:[0,\infty)\times\Omega\to\mathbb{R}$ and for $i=1,...,k$, $D^{\alpha(i)}\hspace*{-0.1cm}\left[u(t,x)\right]$ are partial derivatives as denoted in (\ref{parder}). Assume that solutions to (\ref{th_pde}) exist, are unique and depend continuously on initial conditions.
This implies that the function  is continuously differentiable in $t$ and for each $t\in[0,\infty)$, $u(t,\cdot)\in W^{2,2}(\Omega)$.

\begin{mythm}
\label{th1}
Let there exist continuous $V:L_2(\Omega)\to\R$, $l,m\in\N$ and $b,a>0$  such that
\[
a\|v\|_{L_2}^l\leq V(v)\leq b\|v\|_{L_2}^m,\numberthis{\label{th_V}}\vspace*{-0.05in}
\]
for all $v\in L_2(\Omega)$.
Furthermore, suppose that there exists $c\geq0$ such that for all $t\geq0$ the upper right-hand derivative
\begin{equation}
\label{th_dV}
D^+\left[V(u(t,\cdot))\right]\leq -c\| u(t,\cdot)\|_{L_2}^m,\vspace*{-0.05in}
\end{equation}
where $u$ satisfies (\ref{th_pde}).
Then\vspace*{-0.05in}
\[
\|u(t,\cdot)\|_{L_2}\leq\sqrt[l]{\frac{b}{a}}\|u(0,\cdot)\|_{L_2}^{m/l}\exp\left\{ -\frac{c}{lb}t\right\}\ \text{for all}\ t\geq 0.
\]
\end{mythm}
\begin{proof}
Let conditions of Theorem \ref{th1} be satisfied.
Since $W^{2,2}(\Omega)\subset L_2(\Omega)$, from (\ref{th_V}) it follows that for each $t\geq0$
\[
a\|u(t,\cdot)\|_{L_2}^l\leq V(u(t,\cdot))\leq b\|u(t,\cdot)\|_{L_2}^m.\numberthis{\label{th_ineq_1}}\vspace*{-0.05in}
\]
Dividing both sides of the second inequality in (\ref{th_ineq_1}) by $b$ results in\vspace*{-0.05in}
\[
\frac{1}{b}V(u(t,\cdot))\leq\|u(t,\cdot)\|_{L_2}^m\numberthis{\label{th_ineq_2}}.\vspace*{-0.05in}
\]
After multiplying both sides of (\ref{th_ineq_2}) by $-c$, we have\vspace*{-0.05in}
\[
-c\|u(t,\cdot)\|_{L_2}^m\leq-\frac{c}{b}V(u(t,\cdot)).\numberthis{\label{th_ineq_3}}\vspace*{-0.05in}
\]
From (\ref{th_dV}) and (\ref{th_ineq_3}) it follows that\vspace*{-0.05in}
\[
D^+\left[V(u(t,\cdot))\right]\leq-\frac{c}{b}V(u(t,\cdot)).\numberthis{\label{th_ineq_4}}\vspace*{-0.05in}
\]
To use the comparison principle, consider the ODE
\[
\frac{d}{dt}\left[\phi(t)\right]=-\frac{c}{b}\phi(t),\quad \phi(0)=V(u(0,\cdot)), \numberthis{\label{th_ode}}
\]
where $t\in(0,\infty)$ and function $\phi:[0,\infty)\to\mathbb{R}$ is continuous.
The well-known solution for (\ref{th_ode}) is
\[
\phi(t)=V(u(0,\cdot))\exp\left\{-\frac{c}{b}t\right\}
\]
for all $t\geq0$. Applying Lemma 1 for (\ref{th_ineq_4}) and (\ref{th_ode}) results in
\[
V(u(t,\cdot))\leq V(u(0,\cdot))\exp\left\{-\frac{c}{b}t\right\}\numberthis{\label{th_ineq_5}}
\]
for all $t\geq0$. Substituting $t=0$ in the second inequality of (\ref{th_ineq_1}) implies
\[
V(u(0,\cdot))\leq b\|u(0,\cdot)\|_{L_2}^m.\numberthis{\label{th_ineq_6}}
\]
Combining the first inequality of (\ref{th_ineq_1}) with (\ref{th_ineq_6}) and (\ref{th_ineq_5}) gives
\begin{align*}
a\|u(t,\cdot)\|_{L_2}^l&\leq V(u(t,\cdot))\leq V(u(0,\cdot))\exp\left\{-\frac{c}{b}t\right\}\\
&\leq b\|u(0,\cdot)\|_{L_2}^m\exp\left\{-\frac{c}{b}t\right\}\numberthis{\label{th_ineq_chain1}}.
\end{align*}
Dividing (\ref{th_ineq_chain1}) by $a$ and taking the $l^{th}$ root results in
\[
\|u(t,\cdot)\|_{L_2}\leq\sqrt[l]{\frac{b}{a}}\|u(0,\cdot)\|_{L_2}^{m/l}\exp\left\{-\frac{c}{lb}t\right\}\ \text{for all}\ t\geq 0.
\]
\end{proof}
\vspace*{-0.2in}
%%%%%%%%%%%%%%%%%%%%%%%%%%%%%%%%%%%%%%%%%%%%%%%%%%%%%%%%%%%%%%%%%%%%%%%%%%%%%%%%
\section{STABILITY TEST EXPRESSED AS\\  OPTIMIZATION OVER POLYNOMIALS}
\label{Sec_POP}

In this paper, we focus on PDEs of the following form. For all $t>0$ and $x\in\Omega:=(0,1)^2$,\vspace{-0.05in}
\begin{align*}
u_{t}(t,x)&=a(x)u_{x_1x_1}(t,x)+b(x)u_{x_1x_2}(t,x)\\
&\quad +c(x)u_{x_2x_2}(t,x)+d(x)u_{x_1}(t,x)\\
&\quad +e(x)u_{x_2}(t,x)+f(x)u(t,x),\numberthis\label{pde1}\vspace{-0.05in}
\end{align*}
where $a,b,c,d,e,f\in\mathbb{R}[x]$. Assume that solution to (\ref{pde1}) exists, is unique and depends continuously on initial conditions.
For each $t\geq0$ suppose $u(t,\cdot)\in W^{2,2}(\Omega)$. Furthermore, let $u$ satisfy zero Dirichlet boundary conditions, i.e.
\begin{align*}
\hspace*{-0.2cm} u(t,1,x_2)=u(t,0,x_2)=u(t,x_1,1)=u(t,x_1,0)=\hspace*{-0.05cm} 0\numberthis\label{bc}
\end{align*}
for all $x_1,x_2\in[0,1]$. Define $V:L_2(\Omega)\to\mathbb{R}$ as
\[
V(v):=\int_\Omega s(x)v(x)^2\,dx,\numberthis\label{V}
\]
where $s\in\mathbb{R}[x]$.
Using $u(t,\cdot)$ for $v$ in (\ref{V}) and differentiating the result with respect to $t$ gives
\begin{align*}
\frac{d}{dt}\left[V(u(t,\cdot))\right]&=\frac{d}{dt}\left[\int_\Omega s(x)u(t,x)^2\,dx\right]\\
&=\int_\Omega 2s(x)u(t,x)u_t(t,x)\,dx.\numberthis\label{dV_Leibniz}
\end{align*}
Substituting for $u_t(t,x)$ from (\ref{pde1}) into (\ref{dV_Leibniz}) implies
\begin{align*}
\frac{d}{dt}\left[V(u(t,\cdot))\right]\hspace*{-0.02in}&=\int_\Omega 2s(x)u(t,x)\Big(a(x)u_{x_1x_1}(t,x)\ \ \ \ \ \ \ \ \ \ \ \ \\
&\qquad+b(x)u_{x_1x_2}(t,x)+c(x)u_{x_2x_2}(t,x)\\
&\qquad+d(x)u_{x_1}(t,x)+e(x)u_{x_2}(t,x)\\
&\qquad +f(x)u(t,x)\Big)\,dx\\
&=\hspace*{-0.01in}I_1(t)\hspace*{-0.01in}+\hspace*{-0.01in}I_2(t)\hspace*{-0.01in}+\hspace*{-0.01in}I_3(t)\hspace*{-0.01in}+\hspace*{-0.01in}I_4(t)\hspace*{-0.01in}+\hspace*{-0.01in}I_5(t),\numberthis{\label{dV_I}}\vspace*{-0.1in}
\end{align*} 
where\vspace*{-0.1in}
\begin{align*}
I_1(t):=&\int_\Omega 2s(x)u(t,x)a(x)u_{x_1x_1}(t,x)\,dx,\\
I_2(t):=&\int_\Omega s(x)u(t,x)b(x)u_{x_2x_1}(t,x)\,dx,\\
I_3(t):=&\int_\Omega s(x)u(t,x)b(x)u_{x_1x_2}(t,x)\,dx,\\
I_4(t):=&\int_\Omega 2s(x)u(t,x)c(x)u_{x_2x_2}(t,x)\,dx,\\
I_5(t):=&\int_\Omega 2s(x)u(t,x)\Big(d(x)u_{x_1}(t,x)\\
&\quad\ +e(x)u_{x_2}(t,x)+f(x)u(t,x)\Big)\,dx.\vspace*{-0.1in}
\end{align*}
Note that, based on section 5.2.3 of \cite{c27}, we have\vspace*{-0.05in}
\[
u_{x_1x_2}(t,x)=u_{x_2x_1}(t,x)\numberthis{\label{Evans}}\vspace*{-0.05in}
\]
for all $x\in\Omega$. We used property (\ref{Evans}) to define $I_2$ and $I_3$.

Alternatively, $I_5$ can be formulated as
\[
I_5(t)=\int_\Omega U^T(t,x) Z_5(x) U(t,x)\,dx,\numberthis{\label{I5}}
\]
where for all $x\in\Omega$
\[
U^T(t,x):=\left[u(t,x)\ u_{x_1}(t,x)\ u_{x_2}(t,x)\right]\numberthis{\label{U}}\vspace*{-0.05in}
\]
and\vspace*{-0.05in}
\[
Z_5(x):=\begin{bmatrix} 2s(x)f(x) & s(x)d(x) & s(x)e(x)\\
                      * & 0 & 0\\
                      * & * & 0
\end{bmatrix}.\vspace*{-0.05in}
\]
Using integration by parts and boundary conditions (\ref{bc}), $I_1$ can be rewritten as follows.\vspace*{-0.05in}
\begin{align*}
\hspace*{-0.1cm}I_1(t)&=\int_\Omega 2s(x)u(t,x)a(x)\frac{d}{dx_1}[u_{x_1}(t,x)]\,dx\\
&=2\int_0^1\hspace*{-0.1cm}\Big(s(x)u(t,x)a(x)u_{x_1}(t,x)|_{x_1=0}^{x_1=1}\\
&\qquad\ \ \, -\hspace*{-0.1cm}\int_0^1\hspace*{-0.1cm} u_{x_1}(t,x)\frac{d}{dx_1}[s(x)u(t,x)a(x)]\,dx_1\Big)\,dx_2\\
&=-\int_\Omega 2u_{x_1}(t,x)\Big(u(t,x)\frac{d}{dx_1}[s(x)a(x)]\\
&\qquad\ \ \ +s(x)a(x)u_{x_1}(t,x)\Big)\,dx\\
&=-\int_\Omega U^T(t,x) Z_1(x) U(t,x)\,dx,\numberthis{\label{I1}}\vspace*{-0.2in}
\end{align*}
where for all $x\in\Omega$ \vspace*{-0.05in}
\[
Z_1(x):=\begin{bmatrix} 0 & \frac{d}{dx_1}[s(x)a(x)] & 0\\
                      * & 2s(x)a(x) & 0\\
                      * & * & 0
\end{bmatrix}.%\numberthis{\label{Z1}}
\]
Following steps of (\ref{I1}) for $I_2, I_3$ and $I_4$, we get
\begin{align*}
\hspace*{-0.2cm}I_2(t)&=\int_\Omega s(x)u(t,x)b(x)\frac{d}{dx_1}[u_{x_2}(t,x)]\,dx\\
&=\int_0^1\hspace*{-0.1cm}\Big(s(x)u(t,x)b(x)u_{x_2}(t,x)|_{x_1=0}^{x_1=1}\\
&\qquad\ \ -\int_0^1\hspace*{-0.1cm} u_{x_2}(t,x)\frac{d}{dx_1}[s(x)u(t,x)b(x)]\,dx_1\Big)dx_2\\
&=-\int_\Omega U^T(t,x) Z_2(x) U(t,x)\,dx,
\end{align*}
\begin{align*}
I_3(t)&=\int_\Omega s(x)u(t,x)b(x)\frac{d}{dx_2}[u_{x_1}(t,x)]\,dx\\
&=\int_0^1\Big(s(x)u(t,x)b(x)u_{x_1}(t,x)|_{x_2=0}^{x_2=1}\\
&\qquad\ \ \ -\int_0^1 u_{x_1}(t,x)\frac{d}{dx_2}[s(x)u(t,x)b(x)]dx_2\Big)\,dx_1\\
&=-\int_\Omega U^T(t,x) Z_3(x) U(t,x)\,dx,\\
I_4(t)&=\int_\Omega 2s(x)u(t,x)c(x)\frac{d}{dx_2}[u_{x_2}(t,x)]\,dx\\
&=-\int_\Omega U^T(t,x) Z_4(x) U(t,x)\,dx,\numberthis{\label{I4}}
\end{align*}
where $U$ is defined as in (\ref{U}) and for all $x\in\Omega$
\begin{align*}
Z_2(x)&:=\begin{bmatrix} 0 & 0 & \frac{1}{2}\frac{d}{dx_1}[s(x)b(x)]\\
                      * & 0 & \frac{1}{2}s(x)b(x)\\
                      * & * & 0
\end{bmatrix},\\ %\numberthis{\label{Z2}}
Z_3(x)&:=\begin{bmatrix} 0 & \frac{1}{2}\frac{d}{dx_2}[s(x)b(x)] & 0\\
                      * & 0 & \frac{1}{2}s(x)b(x)\\
                      * & * & 0
\end{bmatrix},\\ %\numberthis{\label{Z3}}
Z_4(x)&:=\begin{bmatrix} 0 & 0 & \frac{d}{dx_2}[s(x)c(x)]\\
                      * & 0 & 0\\
                      * & * & 2s(x)c(x)
\end{bmatrix}.%\numberthis{\label{Z4}}
\end{align*}
By combining (\ref{dV_I}), (\ref{I5}), (\ref{I1}) and (\ref{I4}) it follows that
\[
\frac{d}{dt}[V(u(t,\cdot))]=\int_\Omega U^T(t,x) Q(x) U(t,x)\,dx,\numberthis{\label{dV_Q}}
\]
where for all $x\in\Omega$
\[
Q(x):=\begin{bmatrix} 2s(x)f(x) & Q_1(x) & Q_2(x)\\
                      * & -2s(x)a(x) & -s(x)b(x)\\
                      * & * & -2s(x)c(x)
\end{bmatrix}\numberthis{\label{Q}}
\]
with
\begin{align*}
Q_1(x):&=s(x)d(x)-\frac{d}{dx_1}[s(x)a(x)]-\frac{1}{2}\frac{d}{dx_2}[s(x)b(x)],\\
Q_2(x):&=s(x)e(x)-\frac{1}{2}\frac{d}{dx_1}[s(x)b(x)]-\frac{d}{dx_2}[s(x)c(x)].
\end{align*}

\subsection{Spacing functions}

If $Q(x)\leq0$ for all $x\in\Omega$, then the time derivative in (\ref{dV_Q}) is clearly non-positive for all $t>0$.
However, such a condition on $Q$ is conservative.
To decrease that conservatism, we introduce matrix valued functions $\Upsilon_i$ such that
\[
\int_\Omega U^T(t,x)\Upsilon_i(x) U(t,x)\,dx=0
\]
and, therefore, can be added to $Q$ without altering the integral. $\Upsilon_i$ are called spacing functions. We parameterize $\Upsilon_i$ by polynomials $p_i$. The idea came from \cite{c25}.

The following holds for any $p_1\in\mathbb{R}[x]$, because of the boundary conditions (\ref{bc}).
\begin{align*}
\int_\Omega\frac{d}{dx_1}&[u(t,x)p_1(x)u(t,x)]\,dx\\
&\ \ =\int_{0}^{1}\Big(u(t,x)p_1(x)u(t,x)|_{x_1=0}^{x_1=1}\Big)\,dx_2=0\numberthis{\label{Int_p1_0}}.
\end{align*}
Using the chain rule, we have
\begin{align*}
\frac{d}{dx_1}[&u(t,x)p_1(x)u(t,x)]\\
&=u(t,x)\frac{d}{dx_1}[p_1(x)]u(t,x)+2p_1(x)u(t,x)u_{x_1}(t,x)\\
&=U^T(t,x)\Upsilon_1(x)U(t,x)\numberthis{\label{Der_with_Ups_1}},
\end{align*}
where for all $x\in\Omega$
\[
\Upsilon_1(x):=\begin{bmatrix} \frac{d}{dx_1}[p_1(x)] & p_1(x) & 0\\
                      * & 0 & 0\\
                      * & * & 0
\end{bmatrix}.\numberthis{\label{Ups_1}}
\]
Combining (\ref{Int_p1_0}) and (\ref{Der_with_Ups_1}) results in
\[
\int_\Omega U(t,x)^T\Upsilon_1(x)U(t,x)\,dx=0\numberthis{\label{Int_Ups_1}}.
\]
Likewise in (\ref{Int_p1_0}), because of the boundary conditions (\ref{bc}), the following holds for any $p_2\in\mathbb{R}[x]$.
\begin{align*}
\int_\Omega\frac{d}{dx_2}[u&(t,x)p_2(x)u(t,x)]\,dx=0.
\end{align*}
Following steps of (\ref{Der_with_Ups_1}) for $\frac{d}{dx_2}[u(t,x)p_2(x)u(t,x)]$, gives
\[
\Upsilon_2(x):=\begin{bmatrix} \frac{d}{dx_2}[p_2(x)] & 0 & p_2(x)\\
                      * & 0 & 0\\
                      * & * & 0
\end{bmatrix}\numberthis{\label{Ups_2}}
\]
such that
\[
\int_\Omega U(t,x)^T\Upsilon_2(x)U(t,x)\,dx=0\numberthis{\label{Int_Ups_2}}.
\]
Similarly to (\ref{Int_p1_0}), the following is true for any $p_3\in\mathbb{R}[x]$.
\begin{align*}
\int_\Omega \frac{d}{dx_2}[u&(t,x)p_3(x) u_{x_1}(t,x)]\,dx=0\numberthis{\label{Int_p3_0}}.
\end{align*}
Note that the left-hand side of (\ref{Int_p3_0}) can be written as follows.
\begin{align*}
\int_\Omega \frac{d}{dx_2}&\left[u(t,x)p_3(x)u_{x_1}(t,x)\right]\,dx\\
=&\int_\Omega \Big(u_{x_2}(t,x) p_3(x)u_{x_1}(t,x)\\
&\hspace*{0.5cm}+u(t,x)\frac{d}{dx_2}[p_3(x)]u_{x_1}(t,x)\Big)dx\\
&\hspace*{1.5cm}+\int_\Omega u(t,x)p_3(x) u_{x_2x_1}(t,x)\,dx\numberthis{\label{Int_p3_1}},
\end{align*}
where we need property (\ref{Evans}).
Applying integration by parts to the second integral of the right-hand side of the last equation in (\ref{Int_p3_1}) results in
\begin{align*}
\int_\Omega& u(t,x)p_3(x)\frac{d}{dx_1}\left[u_{x_2}(t,x)\right]\,dx\\
=&\int_0^1\Big( u(t,x)p_3(x)u_{x_2}(t,x)|_{x_1=0}^{x_1=1}\\
&-\int_0^1 u_{x_2}(t,x) \frac{d}{dx_1}\left[u(t,x)p_3(x)\right]\,dx_1\Big)\,dx_2\\
=&-\hspace*{-0.1cm}\int_\Omega\hspace*{-0.1cm} u_{x_2}(t,x)\Big(u_{x_1}(t,x)p_3(x)+u(t,x)\frac{d}{dx_1}[p_3(x)]\Big)\,dx.\numberthis{\label{Int_p3_2}}\vspace{-0.6in}
\end{align*}
From (\ref{Int_p3_1}) and (\ref{Int_p3_2}) the following holds.
\begin{align*}
\int_\Omega \frac{d}{dx_2}&[u(t,x)p_3(x)u_{x_1}(t,x)]\,dx\ \ \ \ \ \ \ \ \ \ \ \ \ \ \ \ \ \ \ \ \ \ \ \ \ \ \ \ \ \ \\
=&\int_\Omega \Big(u(t,x)\frac{d}{dx_2}[p_3(x)]u_{x_1}(t,x)\\
&\hspace*{2cm}-u(t,x)\frac{d}{dx_1}[p_3(x)]u_{x_2}(t,x)\Big)dx\\
=&\int_\Omega U(t,x)^T\Upsilon_3(x)U(t,x)\,dx,\numberthis{\label{int_Ups_3}}
\end{align*}
where for all $x\in\Omega$
\[
\Upsilon_3(x):=\begin{bmatrix} 0 & \frac{1}{2}\frac{d}{dx_2}[p_3(x)]  & -\frac{1}{2}\frac{d}{dx_1}[p_3(x)]\\
                     * & 0 & 0\\
                     * & * & 0
\end{bmatrix}.\numberthis{\label{Ups_3}}
\]
Combining (\ref{Int_p3_0}) and (\ref{int_Ups_3}) gives
\[
\int_\Omega U(t,x)^T\Upsilon_3(x)U(t,x)\,dx=0\numberthis{\label{Int_Ups_3}}.
\]
Following steps (\ref{Int_p3_0}) -- (\ref{int_Ups_3}) for $\frac{d}{dx_1}[u(t,x)p_4(x) u_{x_2}(t,x)]$ with any $p_4\in\mathbb{R}[x]$, leads to the following.
\[
\int_\Omega U(t,x)^T\Upsilon_4(x)U(t,x)\,dx=0\numberthis{\label{Int_Ups_4}},
\]
where for all $x\in\Omega$
\[
\Upsilon_4(x):=\begin{bmatrix} 0 & -\frac{1}{2}\frac{d}{dx_2}[p_4(x)]  & \frac{1}{2}\frac{d}{dx_1}[p_4(x)]\\
                     * & 0 & 0\\
                     * & * & 0
\end{bmatrix}\numberthis{\label{Ups_4}}.
\]
From (\ref{dV_Q}), (\ref{Int_Ups_1}), (\ref{Int_Ups_2}), (\ref{Int_Ups_3}) and (\ref{Int_Ups_4}) the following holds.
\[
\hspace*{-0.4cm}\frac{d}{dt}[V(u(t,\cdot))]\hspace*{-0.1cm}=\hspace*{-0.2cm}\int_\Omega\hspace*{-0.1cm} U^T(t,x) \hspace*{-0.1cm}\left(\hspace*{-0.1cm}Q(x)\hspace*{-0.1cm}+\hspace*{-0.1cm}\sum_{i=1}^4\Upsilon_i(x)\hspace*{-0.1cm}\right)\hspace*{-0.1cm} U(t,x)dx.\hspace*{-0.3cm}\numberthis{\label{dV4}}
\]
By substituting for $Q$ and $\Upsilon_i$ from (\ref{Q}), (\ref{Ups_1}), (\ref{Ups_2}), (\ref{Ups_3}) and (\ref{Ups_4}) in (\ref{dV4}) we can define
\[
M=\Phi(a,b,c,d,e,f,s,p_1,p_2,p_3,p_4)\numberthis{\label{M}},
\]
if for all $x\in\Omega$
\[
M(x)=\begin{bmatrix}
M_1(x) & M_2(x) & M_3(x)\\
* & -2s(x)a(x) & -s(x)b(x)\\
* & * & -2s(x)c(x)
\end{bmatrix}\numberthis{\label{matrixM}},
\]
\vspace{-0.3cm}

where
\begin{align*}
M_1(x)&:=2s(x)f(x)+\frac{d}{dx_1}[p_1(x)]+\frac{d}{dx_2}[p_2(x)],\\
M_2(x)&:=s(x)d(x)-\frac{d}{dx_1}[s(x)a(x)]-\frac{1}{2}\frac{d}{dx_2}[s(x)b(x)]\\
&\quad+p_1(x)+\frac{1}{2}\frac{d}{dx_2}[p_3(x)-p_4(x)],\\
M_3(x)&:=s(x)e(x)-\frac{1}{2}\frac{d}{dx_1}[s(x)b(x)]-\frac{d}{dx_2}[s(x)c(x)]\\
&\quad+p_2(x)+\frac{1}{2}\frac{d}{dx_1}[p_4(x)-p_3(x)],\numberthis{\label{elementsofM}}
%M_4(x)&:=-2s(x)a(x),\\
%M_5(x)&:=-s(x)b(x),\\
%M_6(x)&:=-2s(x)c(x),
\end{align*}
\vspace{-0.6cm}

such that
\[
\frac{d}{dt}[V(u(t,\cdot))]=\int_\Omega U^T(t,x) M(x) U(t,x)dx\numberthis{\label{dVM}}.
\]
\begin{mythm}
\label{thpoly}
Suppose that for (\ref{pde1}) there exist $s,p_i\in\mathbb{R}[x]$, for $i=1,...,4$, and $\theta>0$, such that $s(x)\geq\theta$ and $M(x)\leq0$ for all $x\in\Omega$, where $M$ is defined as in (\ref{M}) -- (\ref{elementsofM}). Then (\ref{pde1}) is stable.

\end{mythm}
\begin{proof}
If conditions of Theorem~\ref{thpoly} are satisfied, let
\[
a=\inf_{x\in\Omega}\{ s(x)\},\ b=\sup_{x\in\Omega} \{s(x)\}.\numberthis{\label{ab}}
\]
Since $s(x)\geq\theta>0$ for all $x\in\Omega$, then $b,a>0$ and the following holds for all $v\in L_2(\Omega)$.
\begin{align*}
a\|v\|^2_{L_2}&=\inf_{x\in\Omega} \{s(x)\}\int_\Omega v^2(x)\,dx\leq\int_\Omega s(x)v^2(x)\,dx\\
&\leq \sup_{x\in\Omega} \{s(x)\}\int_\Omega v^2(x)\,dx=b\|v\|^2_{L_2}.\numberthis{\label{ineqthplo}}
\end{align*}
Using (\ref{V}) we have $a\|v\|^2_{L_2}\leq V(v)\leq b\|v\|^2_{L_2}$.
Since $M(x)\leq 0$, from (\ref{dVM}) it follows that $\frac{d}{dt}[V(u(t,\cdot))]\leq0$ for all $t>0$. Theorem \ref{th1} with $a,b$, defined as in (\ref{ab}), $m=l=2$ and $c=0$ ensures stability of (\ref{pde1}).
\end{proof}

\begin{mythm}
\label{thpoly2}
Suppose that for (\ref{pde1}) there exist $\theta,\gamma>0$ and $s,p_i\in\mathbb{R}[x]$, for $i=1,...,4$, such that $s(x)\geq\theta$ and $M(x)+\gamma S(x)\leq0$ for all $x\in\Omega$, where $M$ is defined as in (\ref{M}) -- (\ref{elementsofM}) and
 \[
S(x):=\begin{bmatrix} s(x) & 0& 0\\
                         *&0&0\\
                         *&*&0
      \end{bmatrix}\numberthis{\label{matrix_S}}.
\]
Then for all $t>0$ solution to (\ref{pde1}) satisfies
\[
\|u(t,\cdot)\|_{L_2}\leq \sqrt{\frac{b}{a}}\|u(0,\cdot)\|_{L_2}\exp\{-\frac{\gamma}{2} t\},\numberthis{\label{solthpoly2}}
\]
where $a,b$ are defined as in (\ref{ab}).
\end{mythm}
\begin{proof}
\hspace*{-0.05in} Under the assumptions of Theorem \ref{thpoly2}, (\ref{ineqthplo})~holds. With (\ref{matrix_S}) and (\ref{U}) we can write
\[
V(u(t,\cdot))=\int_\Omega U(t,x)^T S(x)U(t,x)\,dx\numberthis{\label{V2}}.
\]
Since $M(x)+\gamma S(x)\leq0$ for all $x\in\Omega$, it holds that
\[
\int_\Omega U^T(t,x) (M(x)+\gamma S(x)) U(t,x)\,dx\leq0.\numberthis{\label{ineq_M_S}}
\]
Since $\gamma$ is a scalar, (\ref{ineq_M_S}) can be easily satisfied as follows.
\begin{align*}
\int_\Omega\hspace*{-0.1cm} U^T(t,x) M(x) U(t,x)dx\leq\hspace*{-0.1cm}-\gamma\hspace*{-0.1cm}\int_\Omega\hspace*{-0.1cm} U(t,x)^T S(x)U(t,x)dx,
\end{align*}
which with (\ref{dVM}) and (\ref{V2}) provides $\frac{d}{dt}[V(u(t,\cdot))]\leq-\gamma V(u(t,\cdot))$.
Using proof of Theorem \ref{th1} with $c/b=\gamma$, results in (\ref{solthpoly2}).
\end{proof}\vspace{-0.01in}

%%%%%%%%%%%%%%%%%%%%%%%%%%%%%%%%%%%%%%%%%%%%%%%%%%%%%%%%%%%%%%%%%%%%%%%%%%%%%%%%
\section{LYAPUNOV STABILITY IN TERMS OF\\ SOS OPTIMIZATION}

\label{Sec_SOS}
Solving optimization over polynomials is computationally hard.
Thus, in this section we present SOS programming alternatives, whose solutions yield solutions to problems in Theorems \ref{thpoly} and \ref{thpoly2}  and, therefore, provide sufficient conditions for stability of System (\ref{pde1}).

\begin{mythm}
\label{sosth}
Suppose that for (\ref{pde1}) there exist $s,p_i\in\R[x]$ for $i=1,...,4$, $\theta>0$, $n_1,n_2\in\Sigma[x]$ and $Q_1,Q_2,Q_3\in\Sigma[\s^3(x)]$ such that for all $x_1,x_2\in(0,1)$
\begin{align*}
s(x)=&\,\theta+x_1(1-x_1)n_1(x)+x_2(1-x_2)n_2(x),\\
M(x)=&-Q_1(x)-x_1(1-x_1)Q_2(x)\\
&-x_2(1-x_2)Q_3(x),\numberthis{\label{SOSineq1}}
\end{align*}
where $M$ is defined as in (\ref{M}) -- (\ref{elementsofM}).
Then (\ref{pde1}) is stable.
\end{mythm}

\begin{proof}
If (\ref{SOSineq1}) holds, then clearly $s(x)\geq\theta$ and $M(x)\leq~0$ for all $x\in\Omega$.
Using Theorem \ref{thpoly} provides stability of (\ref{pde1}).
\end{proof}

\begin{mythm}
\label{sosth2}
Suppose that for (\ref{pde1}) there exist $\theta,\gamma>0$, $s,p_i\in\R[x]$ for $i=1,...,4$, $n_1,n_2\in\Sigma[x]$ and $Q_4,Q_5,Q_6\in\Sigma[\s^3(x)]$ such that for all $x_1,x_2\in(0,1)$
\begin{align*}
s(x)=&\,\theta+x_1(1-x_1)n_1(x)+x_2(1-x_2)n_2(x),\\
M(x)+\gamma S(x)=&-Q_4(x)-x_1(1-x_1)Q_5(x)\\
&-x_2(1-x_2)Q_6(x),\numberthis{\label{SOSineq2}}
\end{align*}
where $M$ is defined as in (\ref{M}) -- (\ref{elementsofM}) and $S$ as in (\ref{matrix_S}), then for all $t>0$ solution to (\ref{pde1}) satisfies
\[
\|u(t,\cdot)\|_{L_2}\leq \sqrt{\frac{b}{a}}\|u(0,\cdot)\|_{L_2}\exp\{-\frac{\gamma}{2} t\},\numberthis{\label{SOSineq3}}
\]
where $a,b$ are defined as in (\ref{ab}).
\end{mythm}
\begin{proof}
If (\ref{SOSineq2}) is true, then for all $x\in\Omega$, $s(x)\geq\theta$ and $M(x)+\gamma S(x)\leq0$, which, combined with Theorem \ref{thpoly2}, gives (\ref{SOSineq3}).
\end{proof}

%%%%%%%%%%%%%%%%%%%%%%%%%%%%%%%%%%%%%%%%%%%%%%%%%%%%%%%%%%%%%%%%%%%%%%%%%%%%%%%%
\section{NUMERICAL VERIFICATION OF THE PROPOSED METHOD}

\subsection{KISS model}

We applied the method described in Sections (\ref{Sec_POP}) and (\ref{Sec_SOS}) to stability of the biological KISS PDE named after Kierstead, Slobodkin and Skelam, which describes population growth on a finite area.
For more details see \cite{c28}.
The system is modeled by the following PDE.
\begin{align*}
u_t(t,x)= h \left(u_{x_1x_1}(t,x)+u_{x_2x_2}(t,x)\right)+ru(t,x),\numberthis{\label{kiss}}
\end{align*}
where $h,r>0$, $x\in\Omega\subset\mathbb{R}^2$ and scalar function $u$ satisfies zero Dirichlet boundary conditions.

It is claimed in \cite{c28} that if $\Omega$ is a square with edge of length $l$, then
\[
l_{cr}:=\sqrt{2\pi^2(\frac{h}{r})}\numberthis{\label{l_cr}}
\]
defines a critical length.
That means, if $l>l_{cr}$, then (\ref{kiss}) is unstable.
Alternatively, for given $l$ and $r$ (\ref{l_cr}) defines $h_{cr}$ as
\[
h_{cr}:=l^2r/2\pi^2.\numberthis{\label{D_cr}}
\]
Therefore, if $h<h_{cr}$, then (\ref{kiss}) is unstable.

For our purpose we fix $l=1$ and arbitrarily choose $r=4$.
Thus, according to (\ref{D_cr}), $h_{cr}\approx 0.203$.

Using a bisection search over $h$, we determine minimum $h_{cr}$ for which SOS problem in Theorem \ref{sosth}, with
\begin{align*}
&a(x)=h,\quad b(x)=0,\quad c(x)=h,\\
&d(x)=0, \quad e(x)=0, \quad f(x)=4\quad\text{for all}\ x\in\Omega\numberthis{\label{kisscoef3}}
\end{align*}
may be shown to be feasible. Results for different degrees of $s$ (deg(s)) are presented in Table (\ref{table_kiss_1}).

Now we choose $l=1$, $h=2$ and $r=4$.
Using a bisection search over $\gamma$, we determine the maximum $\gamma$ for which the SOS problem in Theorem \ref{sosth2}, with (\ref{kisscoef3}), may be shown to be feasible.
Results for different deg($s$) are presented in Table (\ref{table_kiss_2}). Using finite difference scheme, we numerically solve (\ref{kiss}) with $u(0,x)=10^3x_1x_2(1-x_1)(1-x_2)$.
Plots of $\log_{10}\left(\|u(t,\cdot)\|_{L_2}\right)$ versus $t$, using a numerical solution, and bounds on $\log_{10}\left(\|u(t,\cdot)\|_{L_2}\right)$, given by the proposed method for different deg$(s)$, are presented in Fig.~(\ref{kiss_pic}). These plots allow us to determine $\gamma$ by examining the rate of decrease in the $L_2$ norm.
Plots are aligned at $t=0$ in order to better compare our SOS estimates of $\gamma$ to the estimate of $\gamma$ derived from numerical simulation as a function of increasing deg($s$).\vspace{-0.2in}
\begin{figure}[t]
      \centering
      \hspace*{-0.5cm}\includegraphics[scale=0.19]{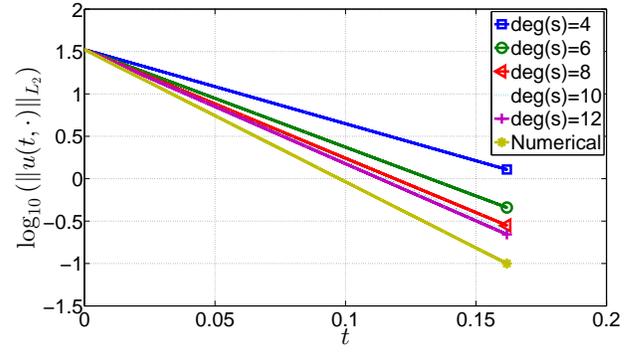}
      \caption{Semi-log plots of the $L_2$ norm of the numerical solution to (\ref{kiss}) with $u(0,x)=10^3x_1x_2(1-x_1)(1-x_2)$ and bounds, given by the proposed method with different deg($s)$.}\vspace*{-0.1in}
      \label{kiss_pic}
\end{figure}
\begin{table}
\caption{Minimum $h_{cr}$ vs $\text{deg}(s$) for (\ref{kiss})}\vspace*{-0.2in}
\label{table_kiss_1}
\begin{center}
\begin{tabular}{|c||c||c||c||c||c||c|}
\hline
deg($s$) &     4 &     6 &      8&    10  & 12 & analytic\\
\hline
$h_{cr}$ & 0.332 & 0.259 & 0.238 & 0.229  & 0.227 & 0.203\\
\hline
\end{tabular}
\end{center}\vspace*{-0.2in}
\end{table}
\begin{table}\vspace*{-0.2in}
\caption{Maximum $\gamma$ vs deg($s$) for (\ref{kiss}) with $h=2$}\vspace*{-0.2in}
\label{table_kiss_2}
\begin{center}
\begin{tabular}{|c||c||c||c||c||c|}
\hline
deg($s$) &     4 &    6    &  8  & 10 & 12 \\
\hline
$\gamma$ & 40.25 &   53    & 59  & 61 & 62 \\
\hline
\end{tabular}
\end{center}\vspace*{-0.3in}
\end{table}

\subsection{Randomly generated system}
Consider\vspace*{-0.01in}
\begin{align*}
u_t(t,x)=\;&(5x_1^2-15x_1x_2+13x_2^2)(u_{x_1x_1}(t,x)\\
&+u_{x_2x_2}(t,x))+(10x_1-15x_2)u_{x_1}(t,x)\\
&+(-15x_1+26x_2)u_{x_2}(t,x)-(17x_1^4-30x_2\\
&-25x_1^2-8x_2^3-50x_2^4)u(t,x),\\
u(0,x)=\;&10^3x_1x_2(1-x_1)(1-x_2)\numberthis{\label{our1}}\vspace*{-0.01in}
\end{align*}
where $x\in\Omega:=(0,1)^2$ and the scalar function $u$ satisfies zero Dirichlet boundary conditions.

Using a bisection search over $\gamma$, we determine maximum $\gamma$ for which SOS problem in Theorem \ref{sosth2}, with
\begin{align*}
a(x)&=5x_1^2-15x_1x_2+13x_2^2,\quad e(x)=-15x_1+26x_2,\\
c(x)&=5x_1^2-15x_1x_2+13x_2^2,\quad d(x)=10x_1-15x_2,\\
f(x)&=-(17x_1^4-30x_2-25x_1^2-8x_2^3-50x_2^4),\quad b(x)=0
\end{align*}
may be shown to be feasible.

Using finite difference scheme, we numerically solve (\ref{our1}).
The estimated rate of decay, based on numerical solution, is $13.07$.
The computed rate of decay, based on our SOS method, is $12.5$ for deg$(s)=8$.
Plots are given in Fig.~\ref{ex2_pic}
and, as for Fig.~\ref{kiss_pic}, are aligned at $t=0$.

\begin{figure}[t]
      \centering
      \hspace*{-0.5cm}\includegraphics[scale=0.18]{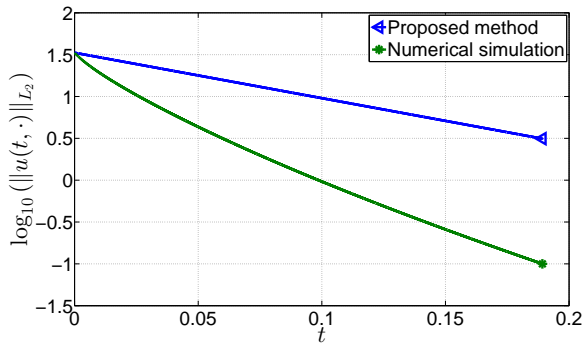}
      \caption{Semi-log plots of the $L_2$ norm of the numerical solution to (\ref{our1}) and bound, given by the proposed method with deg($s)=8$.}\vspace*{-0.25in}
      \label{ex2_pic}
\end{figure}

\section{Conclusion and future work}
In this paper we have presented a method which allows us to search for Lyapunov functionals for parabolic linear PDEs with two spatial variables and spatially dependent coefficients.
We demonstrated accuracy of the method in estimating critical diffusion for the KISS PDE and the rate of decay for the KISS PDE and for a randomly chosen PDE.
In future work, we will extend the method by considering Lyapunov functionals of more complicated types, e.g. functionals with semi-separable kernels, as in \cite{c29}.
In addition, we will consider alternative boundary conditions, as well as semi-linear and nonlinear 2D and 3D parabolic PDEs.


\begin{thebibliography}{99}

\bibitem{c1}
P. D. Christofides, ``Nonlinear and robust control of PDE systems: Methods and applications to transport-reaction processes," {\it Springer Science \& Business Media}, New York, 2001.

\bibitem{c2}
R. F. Curtain and H. Zwart, ``An introduction to infinite-dimensional linear systems theory," {\it Springer Science \& Business Media}, New York, 1995.

\bibitem{c3}
N. H. El-Farra, A. Armaou and P. D. Christofides, ``Analysis and control of parabolic PDE systems with input constraints," {\it Automatica}, volume~39, no.~4, pages 715--725, 2003.

\bibitem{c4}
J. Baker and P. D. Christofides, ``Finite-dimensional approximation and control of non-linear parabolic PDE systems," {\it International Journal of Control}, volume~73, no.~5, pages 439--456, 2000.

\bibitem{c5}
P. D. Christofides and P. Daoutidis, ``Finite-dimensional control of parabolic PDE systems using approximate inertial manifolds," {\it Proceedings of the 36th IEEE Conference on Decision and Control}, pages 1068--1073, 1997.

\bibitem{c5p}
R. Kamyar, M. M. Peet and Y. Peet, ``Solving Large-Scale Robust Stability Problems by Exploiting the Parallel Structure of Polya's Theorem," {\it IEEE Transactions on Automatic Control}, volume~58, No.~8, pages 1931--1947, 2013.

\bibitem{c6}
E. Fridman and Y. Orlov, ``Exponential stability of linear distributed parameter systems with time-varying delays," {\it Automatica}, volume~45, no.~1, pages 194--201, 2009.

\bibitem{c7}
E. Fridman, S. Nicaise and J. Valein, ``Stabilization of second order evolution equations with unbounded feedback with time-dependent delay," {\it SIAM Journal on Control and Optimization}, volume~48, no.~8, pages 5028--5052, 2010.

\bibitem{c8}
O. Solomon and E. Fridman, ``Stability and passivity analysis of semilinear diffusion PDEs with time-delays," {\it International Journal of Control}, volume~88, no.~1 pages 180--192, 2015.

\bibitem{c9}
M. Krstic and A. Smyshlyaev, ``Boundary control of PDEs: A course on backstepping designs," {\it Society for Industrial and Applied Mathematics}, Philadelphia, 2008.

\bibitem{c10}
M. Krstic and A. Smyshlyaev, ``Backstepping boundary control for first order hyperbolic PDEs and application to systems with actuator and sensor delays," {\it Systems \& Control Letters}, volume~57, no.~9, pages 750--758, 2008.

\bibitem{c11}
M. Krstic and A. Smyshlyaev, ``Adaptive boundary control for unstable parabolic PDEs—Part I: Lyapunov design," {\it IEEE Transactions on Automatic Control}, volume~53, no.~7, pages 1575--1591, 2008.

\bibitem{c12}
A. Smyshlyaev and M. Krstic, ``On control design for PDEs with space-dependent diffusivity or time-dependent reactivity," {\it Automatica}, volume~41, no.~9, pages 1601--1608, 2005.

\bibitem{c13}
A. Smyshlyaev and M. Krstic, ``Adaptive control of parabolic PDEs," {\it Princeton University Press}, Princeton, 2005.

\bibitem{c14}
R. Vazquez and M. Krstic, ``Explicit integral operator feedback for local stabilization of nonlinear thermal convection loop PDEs," {\it Systems \& control letters}, volume~55, no.~8, pages 624--632, 2006.

\bibitem{c15}
C. Xu, E. Schuster, R. Vazquez and M. Krstic, ``Stabilization of linearized 2D magnetohydrodynamic channel flow by backstepping boundary control," {\it Systems \& control letters}, volume~57, no.~10, pages 805--812, 2008.

\bibitem{c16}
G. Valmorbida, M. Ahmadi and A. Papachristodoulou, ``Semi-definite programming and functional inequalities for Distributed Parameter Systems," {\it Proceedings of the 53rd IEEE Conference on Decision and Control}, pages 4304--4309, 2014.

\bibitem{c17}
A. Papachristodoulou and M. M. Peet, ``On the analysis of systems described by classes of partial differential equations," {\it Proceedings of the 45th IEEE Conference on Decision and Control}, pages 747--752, 2006.

\bibitem{c18}
M. Ahmadi, G. Valmorbida and A. Papachristodoulou, ``Input-Output Analysis of Distributed Parameter Systems Using Convex Optimization," {\it Proceedings of the 53rd IEEE Conference on Decision and Control}, pages 4310--4315, 2014.

\bibitem{c19}
A. Gahlawat and M. M. Peet, ``Designing observer-based controllers for pde systems: A heat-conducting rod with point observation and boundary control," {\it Proceedings of the 50th IEEE Conference on Decision and Control and European Control Conference}, pages 6985--6990, 2011.

\bibitem{c20}
A. Gahlawat, M. M. Peet and E. Witrant, ``Control and verification of the safety-factor profile in tokamaks using sum-of-squares polynomials," {\it Proceedings of the 18th IFAC World Congress}, 2011.

\bibitem{c21}
A. Gahlawat and M. M. Peet, ``A Convex Approach to Output Feedback Control of Parabolic PDEs Using Sum-of-Squares," {\it arXiv preprint arXiv:1408.5206}, 2014.

\bibitem{c22}
M. Laurent, ``Sums of squares, moment matrices and optimization over polynomials," {\it Emerging applications of algebraic geometry}, pages 157--270, Springer, New York, 2009.

\bibitem{Reza_DCDS2015}
R. Kamyar and M. M. Peet, ``Polynomial Optimization with Applications to Stability Analysis and Control - Alternatives to Sum of Squares," {\it Discrete and Continuous Dynamical Systems Series B}, volume 20, no. 8, pages 2383--2417, 2015.

\bibitem{c23}
P. J. C. Dickinson and J. Povh, ``On an extension of Polya’s Positivstellensatz," {\it Journal of Global Optimization}, volume~61, pages 615--625, 2014.

\bibitem{c24}
S. Prajna, A. Papachristodoulou, P. Seiler and P. A. Parrilo, ``New developments in sum of squares optimization and SOSTOOLS," {\it Proceedings of the American Control Conference}, pages 5606--5611, 2004.

\bibitem{c25}
M. Peet, A. Papachristodoulou, S. Lall, ``Positive forms and stability of linear time-delay systems," {\it SIAM Journal on Control and Optimization}, volume~47, no.~6, pages 3237--3258, 2009.

\bibitem{c26}
H. K. Khalil and J. W. Grizzle, ``Nonlinear systems," {\it Prentice hall}, New Jersey, 1996.

\bibitem{c27}
L. Evans, ``Partial differential equations," {\it American Mathematical Society}, 1998.

\bibitem{c28}
E. E. Holmes, M. A. Lewis, J. E. Banks and R. R. Veit, ``Partial differential equations in ecology: spatial interactions and population dynamics," {\it Ecology}, pages 17--29, 1994.

\bibitem{c29}
M. Peet and A. Papachristodoulou, ``Using polynomial semi-separable kernels to construct infinite-dimensional Lyapunov functions," {\it Proceedings of the 47th IEEE Conference on Decision and Control}, pages 847--852, 2008.
\end{thebibliography}
\end{document}